\documentclass[12pt]{amsart}
\usepackage{fullpage}
\usepackage{amsthm, amsmath, amssymb,amscd}
\usepackage{xcolor,hyperref}
\usepackage{enumitem}
\usepackage[numbers,sort&compress]{natbib}

\theoremstyle{plain}
\newtheorem{theorem}{Theorem}[section]
\newtheorem{lemma}[theorem]{Lemma}

\newtheorem{corollary}[theorem]{Corollary}
\newtheorem{conjecture}[theorem]{Conjecture}
\theoremstyle{definition}
\newtheorem{proposition}[theorem]{Proposition}
\newtheorem{definition}[theorem]{Definition}
\newtheorem{example}[theorem]{Example}

\theoremstyle{remark}
\newtheorem{remark}[theorem]{Remark}

\numberwithin{equation}{section}

\usepackage{tikz}
\tikzstyle{vertex}=[circle, draw, inner sep=0pt, minimum size=3pt]
\newcommand{\vertex}{\node[vertex]}

\allowdisplaybreaks
\begin{document}
	
	%
	%
	%
	%
	%
	%
	%
	%
	%

	\title[ On the Monogenity of Polynomials with Non-Squarefree Discriminants ]
	{On the Monogenity of Polynomials with Non-Squarefree Discriminants }

	
	\author{Rupam Barman}
	\address{Department of Mathematics, Indian Institute of Technology Guwahati, Assam, India, PIN- 781039}
	\email{rupam@iitg.ac.in}

	\author{Anuj Narode}
	\address{Department of Mathematics, Indian Institute of Technology Guwahati, Assam, India, PIN- 781039}
	\email{anujanilrao@iitg.ac.in}
	
	\author{Vinay Wagh}
	\address{Department of Mathematics, Indian Institute of Technology Guwahati, Assam, India, PIN- 781039}
	\email{vinay\_wagh@yahoo.com}
	
	\date{\today}
	
	\thanks{}
	
	
	
	\dedicatory{}
	\begin{abstract}
		In 2012, for any integer $n \ge 2$, Kedlaya constructed an infinite class of monic irreducible polynomials of degree $n$ with integer coefficients having squarefree discriminants. Such polynomials are necessarily monogenic. Further, by extending Kedlaya's approach, for any odd prime $q$, Jones constructed a class of degree $q$ polynomials with non-squarefree discriminants. In this article,  using a similar method provided by Jones, we present another infinite class of monogenic polynomials of degree $q$ with non-squarefree discriminants, where $q$ is a prime of the form $ q = q_0 + q_1 - 1 $, with $ q_0 $ and $ q_1 $ being prime numbers. In addition to this we present a class of non-monogenic polynomials whose coefficients are Sterling numbers of the first kind.
	\end{abstract}
	\maketitle
	\section{Introduction}
Let $f(x) \in \mathbb{Z}[x]$ be a monic irreducible polynomial with a root $\theta$ and $K = \mathbb{Q}(\theta)$. We denote by $\mathbb{Z}_K$ the ring of integers of $K$. The index of a polynomial  $\operatorname{ind}(f(x))$ is the index of $\mathbb{Z}[\theta]$ in $\mathbb{Z}_K$, i.e. $\operatorname{ind}(f(x))=[ \mathbb{Z}_K : \mathbb{Z}[\theta] ].$ Note that, $\mathbb{Z}_K$ is a 
free $\mathbb{Z}$-module of rank $\deg(f)$, and $\mathbb{Z}[\theta]$ is a submodule of $\mathbb{Z}_K$ of the same rank so that $[ \mathbb{Z}_K : \mathbb{Z}[\theta] ]$ is finite.
The polynomial $f(x)$ is called monogenic if the index of $f(x)$ is one, i.e., $[ \mathbb{Z}_K : \mathbb{Z}[\theta] ] = 1$. Further, the number field $K$ is said to be monogenic if there exists some $\alpha \in \mathbb{Z}_K$ such that $\mathbb{Z}_K = \mathbb{Z}[\alpha]$.
Note that the monogenity of the polynomial $f(x)$ implies the monogenity of the number field $K$. However, the converse is not true. For example, let $K = \mathbb{Q}(\theta)$, where $\theta$ is a root of $f(x)=x^2-5$. Then, $K$ is monogenic as $\mathbb{Z}_K = \mathbb{Z}\left[\frac{1 + \sqrt{5}}{2}\right]$. But, $f(x)$ is not monogenic since $[\mathbb{Z}_K: \mathbb{Z}[\sqrt{5}]] = 2$.
\par 
The discriminant  $\triangle(f)$ of the polynomial $f(x)$ is defined as 
\begin{equation}\label{eq-1.1}
	\triangle(f) = (-1)^{n(n-1)/2} \hspace{3pt} a_n^{2n -2} \hspace{3pt} \prod_{i\neq j} (\alpha_i - \alpha_j),
\end{equation} 
where $ n = \deg(f),$ $a_n$ is the leading coefficient of $f(x)$, and $\alpha_i$'s are the roots of $f(x).$ In Section 2, we define $\triangle(f)$ in terms of the resultant. The discriminant $\triangle(K)$ of the number field $K$ and the discriminant of the polynomial $f(x)$ are related by the well-known equation
\begin{equation} \label{eq-2.1}
	\triangle(f) = \operatorname{ind}(f(x))^2 \cdot \triangle(K).	
\end{equation}
Clearly, if $\triangle(f)$ is squarefree, then $f(x)$ is monogenic. However, the converse is not always true.
\par 
One of the classical and important questions in algebraic number theory is to determine whether the ring of integers of an algebraic number field can be generated by a single element i.e., whether it is monogenic. Further, the problem of describing such fields using arithmetic conditions was first posed by Hasse (cf.~\cite{hasse}). Subsequently, a lot of research has been done related to this problem. For example, Ga\'al studied monogenity of number fields having smaller degree (cf. ~\cite{gaal-1}). Jhorar et al. studied the monogenity of binomials in~\cite{jorar}. In \cite{jhakar-1}, Jhakar et al. studied the primes dividing the index of trinomial. Further, monogenity of quadrinomials, reciprocal polynomials, and composition of binomials were studied in \cite{ahmad, jhakar-11, jhakar-2, jorar,jones_2021-1, Jones_2021, fadil}. In addition, Ga\'al's article \cite{gaal-2} provides a bird's-eye view on the developments related to monogenity.
\par 
In 2012, Kedlaya \cite{Kedlaya} provided a construction of infinitely many monic irreducible polynomials of degree $n = r + 2s$ with integer coefficients, possessing squarefree discriminant and exactly $ r $ real roots. These polynomials are necessarily monogenic. Moreover, he proved that their Galois group is the full symmetric group $ S_n $. In 2020, Jones employed Kedlaya's method to construct an infinite class of irreducible polynomials of prime degree $q \geq 3 $ that are monogenic and have non-squarefree discriminant \cite{Jones_2019}. He used the result of Pasten \cite{pasten} to construct an infinite class of monogenic polynomials. In this article, we construct another infinite class of monogenic polynomials, of prime degree $q$ of the form $q = q_0 + q_1 - 1$, having non-squarefree discriminant. To achieve this, we adapt the method of Jones and employ a theorem of Pasten to establish the infinitude of such polynomials. More precisely, we prove the following:
\begin{theorem}{\label{theorem_1}}
	Let $q \geq 3 $ be a prime of the form $q:= q_0 + q_1 - 1 $, where $ q_0 < q_1 $ are primes and let $ d \geq 1 $ be an integer. Let  $m \geq 1 $ be a squarefree integer such that $q \nmid m $ and $ r \mid m$ for all primes $r$ with $q < r < d(q-1) + 1$. Further, let $q_2$ be a prime such that $q_2 \equiv -1 \pmod q$. Define the  polynomials $$G(x):= q a(x) b(x) \hspace{5pt}\text{and}\hspace{5pt} F_0(x) : = \int_{0}^{x} G(t) dt,$$ where
	$ a(x) := \prod_{i=1}^{q_0-1} (x + i q_2 \cdot m \cdot (q-1)!) \hspace{5pt}\text{and}\hspace{5pt} b(x) := \prod_{j=1}^{q_1-1} (x + j  \cdot m \cdot (q-1)!).$
	Then, there exist infinitely many primes $p$ such that $$F(x): = F_0(x) +qmp^d$$ is monogenic and has non-squarefree discriminant.
	Note that the existence of primes $p$ is unconditional when $d < 4 $ and conditional on the $abc$-conjecture for number fields for $d \geq 4$.
\end{theorem}
\begin{remark}
	Observe that in Theorem \ref{theorem_1}, infinitely many primes $q_2$ exist satisfying $q_2 \equiv -1 \pmod q$ due to Dirichlet's theorem on primes in arithmetic progressions.
\end{remark}
\section{preliminaries}
In this section, we state the $abc$-conjecture for number fields, which is a generalization of the classical $abc$-conjecture, and present some preliminary results that will be useful in proving Theorem \ref{theorem_1}. More details on the $abc$-conjecture for number fields can be found in \cite[Section 4]{pasten}).
\begin{conjecture}[$abc$-conjecture for number fields] \label{conj-1}
	Let $K$ be a number field. Let $\epsilon > 0$ and fix mutually distinct elements $b_1, \ldots, b_m \in K$. Let $S$ be a finite set of places of $K$. Then, for all but finitely many $\alpha \in K$, one has 
	$$ (M- 2 \epsilon )h_K(\alpha)< \sum_{i=1}^{M}N_{K,S}^{(1)} ( \alpha - b_i ),$$
	where $h_K(\alpha) $ is the height (relative to K) of $\alpha \in K$ and  $N_{K,S}^{(1)}$ is the truncated counting function.
\end{conjecture}
In Theorem~\ref{theorem_1}, the existence of infinitely many primes $p$ is ensured by the following result, derived from the work of Helfgott and Pasten. In particular, if we have a polynomial that factorizes as a product of distinct irreducible factors, we are interested in determining when there exist infinitely many primes $p$ such that $f(p)$ is squarefree. This is established by the following theorem and its corollary, as stated in \cite[Theorem 2.9 and Corollary 2.10]{Jones_2021}.
\begin{theorem}
	Let $f(x) \in \mathbb{Z}[x]$ be such that $f(x)$ is a product of distinct irreducible factors, where the largest degree of any irreducible factor is $d$. Define $$ N_f(X) = \# \{r \leq X  \mid r \hspace{3pt} \text{is a prime and} \hspace{3pt} f(r) \hspace{3pt} \text{is squarefree} \}.$$
	Then, the following asymptotic holds unconditionally if $d \leq 3$ and holds under the assumption of Conjecture \ref{conj-1} for $f(x)$, if $d \geq 4$: 
	$$ N_f(X) \sim c_f \hspace{2pt} \frac{X}{\log X},$$
	where $$ c_f = \prod_{r \hspace{3pt \text{prime}}} \left( 1 - \frac{\rho_f(r^2)}{r(r-1)}\right)$$ and $\rho_f(r^2)$ is the number of $z \in (\mathbb{Z}/r^2 \mathbb{Z})^* $ such that $f(z) \equiv 0\pmod {r^2}.$
\end{theorem}
\begin{corollary} \label{coro-2.2}
	Let $f(x) \in \mathbb{Z}[x]$ be such that $f(x)$ is a product of distinct irreducible factors, where the largest degree of any irreducible factor is $d$. Further suppose that, for each prime $r$, there exists some $ z \in (\mathbb{Z} /r^2 \mathbb{Z})^{*}$ such that $f(z) \not \equiv 0 \pmod {r^2}.$ If  $d \leq 3$, or if $d \geq 4$ with the assumption of the $abc$-conjecture for number fields for $f(x)$, then there exist infinitely many primes $p$ such that $f(p)$ is squarefree.
\end{corollary}	
\begin{definition}
	Let $q$ be a prime and let $ f(x) = a_0 + a_1 x + a_2x^2 + \cdots + a_nx^n \in \mathbb{Z}[x].$ We say $f(x) $ is $q$-Eisenstein, or Eisenstein with respect to $q$, if $q  \nmid a_n$, $q \mid a_i$ for all $i <n$ and $q^2  \nmid a_0$.
\end{definition}
\begin{theorem}[Eisenstein Criterion] \emph{\cite[Exercise 3.11]{cohen}}
	If  $q$ is  a prime and $f(x) \in \mathbb{Z}[x]$ is a $q$-Eisenstein polynomial, then $f(x)$ is irreducible. 
\end{theorem}
\begin{definition} \label{definition-2}
	The reciprocal of a polynomial $f(x) \in \mathbb{Z}[x]$ is defined as $$\widetilde{f}(x) := x^{\deg(f)} f\left( \frac{1}{x} \right) \in \mathbb{Z}[x].$$
\end{definition} 
The following result is an immediate consequence of Definition \ref{definition-2}.
\begin{proposition} \label{prop-2.7}
	Suppose $f(x) \in \mathbb{Z}[x]$ with $f(0) \neq 0$. Then, $f(x)$ is irreducible if and only if $\widetilde{f}(x)$ is irreducible.
\end{proposition}
\begin{definition}\label{definition_1}
	We define the $n$th elementary symmetric polynomial $e_n$ in variables $ t_1, t_2, \ldots, t_N  $ as 
	$$e_n := 
	\begin{cases}
		\begin{array}{c@{\quad}l}
			1 & \text{if} \hspace{5pt} n = 0, \\
			
			\displaystyle\sum_{1 \leq j_1 < j_2 < \cdots < j_n \leq N}^{} t_{j_1} t_{j_2} \cdots t_{j_n} & \text{if } \hspace{3pt} 1\leq n \leq N ,\\
			
			0 & \text{if } \hspace{3pt} n < 0 \hspace{3pt}\text{or} \hspace{3pt} n > N.
		\end{array}
	\end{cases}
	$$
\end{definition}
\begin{example}
	Suppose that $c(x) = \prod_{i=1}^{3}(x+t_i)$. Then, $c(x)$ can be written in the form $$ c(x) = x^3 + e_1 x^2 + e_2 x + e_3 ,$$ where $e_1 = t_1 + t_2 + t_3$, $e_2 = t_1 t_2 + t_1 t_3 + t_2 t_3$ and $e_3 = t_1 t_2 t_3$.
\end{example}
\begin{proposition}{\label{proposition_1}}
	Suppose that $c(x) = \prod_{i=1}^{n}(x + t_i) = \sum_{i=1}^{n}\alpha_i x^{n-i}$ and $d(x) = \prod_{i=1}^{m}(x + y_i) = \sum_{i=1}^{m}\beta_i x^{m-i}$. Then, $$c(x) \cdot d(x) = \sum_{k=0}^{m+n} \gamma_k x^{m+n-k},$$ where  $\gamma_k = \sum_{l=0}^{k}\alpha_{k-l} \beta_l, \quad  0 \leq k \leq m+n.$
\end{proposition}
\begin{proof}
	Let $k$ be such that $0 \leq k \leq m+n$. It can be observed that for each $0 \leq l \leq k$, $\alpha_{k-l}$ is the co-efficient of $x^{n - k + l}$ in $c(x)$ and $\beta_l$ is the co-efficient of $x^{m - l}$ in $d(x)$. Thus, $\gamma_k = \sum_{l=0}^{k}\alpha_{k-l} \beta_l$ is the co-efficient of $x^{m+n-k}$ in $c(x) \cdot d(x)$. Hence, the result follows. 
\end{proof}
\begin{theorem} \emph{\cite[Theorem 3.6]{conrad} \label{thm-1.4}}
	Let $q$ be a prime and $f(x) \in \mathbb{Z}[x]$ be a monic $q$-Eisenstein polynomial with $\operatorname{deg}(f) = n$. Let $K  = \mathbb{Q}(\alpha)$ with $f(\alpha) = 0$. Then,
	\begin{itemize}
		\item[(1)]$q^{n}\ | \ \triangle(K)$ if $n  \equiv 0\pmod q$, 
		\item[(2)]  $q^{n-1} || \triangle(K)$ if $n \not \equiv 0\pmod q$.
	\end{itemize}
\end{theorem}
Suppose $f(x)$ and $g(x)$ are two polynomials  over an integral domain $\mathcal{R}$, with quotient
field $K$ and let $\overline{K}$ be an algebraic closure of $K$.
\begin{definition} \cite[Section 3.3.2]{cohen}\label{def-1.5}
	Let $f(x) = a (x - \alpha_1)\cdots(x - \alpha_m) $ and $g(x) =  b (x - \beta_1)\cdots(x - \beta_n)$ be the decomposition of $f(x)$ and $g(x)$ in $\overline{K}.$ Then, the resultant $R(f,g)$ of $f(x)$ and $g(x)$ is given by one of the equivalent formulas: 
	\begin{align*}
		R(f,g) &= a^n g(\alpha_1) \cdots g(\alpha_m) \\
		&= (-1)^{mn} b^n f(\beta_1) \cdots f(\beta_n) \\
		&= a^nb^m \prod_{\substack{1 \leq i \leq m \\ 1 \leq j \leq n}} (\alpha_i - \beta_j).	
	\end{align*}
\end{definition}
We now recall another definition of discriminant of a polynomial, which  is equivalent to \eqref{eq-1.1}.
\begin{definition}\cite[Section 3.3.2]{cohen}\label{def-1.6}
	Let $f(x) \in \mathcal{R}[x]$, with $m = \deg(f)$. Then, the discriminant $\triangle(f)$ of $f(x)$ is equal to the expression: $${(-1)^{m(m-1)/2}} R(f(x),f'(x))/{l(f)},$$
	where $f'(x)$ is the derivative of $f(x)$ and $l(f)$ is the leading coefficient of $f(x)$.
\end{definition}	
The next proposition readily follows from Definition \ref{def-1.5}.
\begin{proposition}
	Suppose that $f(x)$ and $g(x)$ are as defined in Definition \ref{def-1.5}. Then, we have $R(f,g) = (-1)^{mn}R(g,f)$.
\end{proposition}
\section{proof of theorem \ref{theorem_1}}
Before proceeding to the proof of Theorem \ref{theorem_1}, we establish two essential lemmas. Lemma~\ref{lemma-1.1} demonstrates that $F_0(x) \in \mathbb{Z}[x]$, while Lemma \ref{lemma-1.2} shows that two numbers arising in the proof of the main theorem are distinct.
\begin{lemma}\label{lemma-1.1}
	The polynomial $ F_0(x) $ as defined in Theorem \ref{theorem_1} has integer coefficients.
\end{lemma}
\begin{proof}
	In Definition \ref{definition_1}, let $ N = q_0 - 1 $ and $ t_i = i q_2 m \cdot (q-1)! $ for each $i$ such that $1 \leq i \leq q_0 - 1$. Then, we can express $a(x)$ as follows:
	\[
	a(x) = x^{q_0 - 1} + f_1 x^{q_0 - 2} + f_2 x^{q_0 - 3} + \cdots + f_{q_0 - 1},
	\]
	where $ f_i's $ are the symmetric polynomials. Similarly, by setting $N = q_1 - 1$ and $ t_j = j m \cdot (q-1)! $ for each $ j $ such that $ 1 \leq j \leq q_1 - 1 $, we can write $ b(x)$ as:
	\[
	b(x) = x^{q_1 - 1} + g_1 x^{q_1 - 2} + g_2 x^{q_1 - 3} + \cdots + g_{q_1 - 1},
	\]
	where $ g_i $ are the symmetric polynomials. By Proposition \ref{proposition_1}, we have
	\[
	a(x) \cdot b(x) = x^{q_0 + q_1 - 2} + e_1 x^{q_0 + q_1 - 3} + e_2 x^{q_0 + q_1 - 4} + \cdots + e_{q_0 + q_1 - 2},
	\]
	where $e_k = \sum_{i = 0}^{k} f_i g_{k - i}$ for each $k$ such that $1 \leq k \leq q_0 + q_1 - 2$. Consequently, we find
	\[
	G(x) = q x^{q_0 + q_1 - 2} + q e_1 x^{q_0 + q_1 - 3} + \cdots + q e_{q_0 + q_1 - 2} x.
	\]
	We note that, for $1 \leq i \leq q_0 + q_1 - 2$, we have $ e_i \equiv 0 \pmod{(q_0 + q_1 - 2)!}$. Therefore, we conclude that $e_{q - (u + 1)} \equiv 0 \pmod{u + 1}$ for $0 \leq u < q - 1$. This concludes that $F_0(x) \in \mathbb{Z}[x]$.
\end{proof}
\begin{lemma}\label{lemma-1.2}
	Suppose $F_0(x)$ is the polynomial as defined in Theorem \ref{theorem_1}. Define $C_i$ and $D_j$ as follows: 
	$$C_i = \frac{F_0 (-i q_2 m \cdot (q-1)!)}{m} \quad \text{and} \quad D_j = \frac{F_0 (-j  m \cdot (q-1)!)}{m}, $$ with $1 \leq i \leq q_0 -1$ and $1 \leq j \leq q_1 -1$. Then, $C_i \neq D_j$ for any $i $ and $j$.
\end{lemma}
\begin{proof}
	Note that,
	\begin{equation}\label{eq-3.6}
		C_i \equiv \frac{(-i q_2 m \cdot (q-1)! )^q }{m} \equiv iq_2 \pmod q \hspace{2pt}.
	\end{equation}
	This shows that all the $C_{i}$'s are distinct. Similarly, we have 
	\begin{equation}\label{eq-3.4}
		\hspace{2pt} D_j \equiv \frac{ (-j  m \cdot (q-1)!)^q}{m}  \equiv j \pmod q.
	\end{equation}
	This shows that $D_j$'s are distinct. Next, we have $q_2 \equiv -1 \pmod{q}$ and $2 \leq i + j \leq q - 1$. Therefore, we obtain  
	$$C_i - D_j \equiv i q_2 - j \pmod{q} = -i - j \pmod{q} \not\equiv 0 \pmod{q}.$$ 
	Hence, $C_i \neq  D_j.$
\end{proof}
We are now in a position to present the proof of Theorem~\ref{theorem_1}.
\begin{proof}[Proof of Theorem \ref{theorem_1}]
	We have the equation $$\triangle(F) = [\mathbb{Z}_K:\mathbb{Z}[\theta] ]^2 \cdot \triangle(K),$$
	where $K = \mathbb{Q}(\theta)$ with $\theta$ a root of $F(x)$. Using this equation our goal is to prove that there are infinitely many primes $p$ such that $\triangle(F) = \triangle(K).$  To establish this we compute $| \triangle(F) | $ and show that the squarefull part of $|\triangle(F)|$ divides $|\triangle(K)|$. But, first we prove that $F(x)$ is  $q$-Eisenstein; in fact, it is Eisenstein with respect to every prime divisor of $m$. 
	\par 
	From Lemma \ref{lemma-1.1}, we have $F_0(X) \in \mathbb{Z}[x]$. Also, using the definition of $F_0(x)$, we obtain
	$$F_0(x) =  x^{q_0 + q_1 -1} +  \frac{e_1 q}{q_0 + q_1 -2} x^{q_0 + q_1 -2} + \frac{e_2 q}{q_0 + q_1 -3} x^{q_0 + q_1 -3} + \cdots + qe_{q_0 + q_1 -2}x.$$
	We have $q = q_0 + q_1 -1$, and hence
	\begin{equation}
		F_0(x) =  x^{q} +  \frac{e_1 q}{q-1} x^{q-1} + \frac{e_2 q}{q-2} x^{q-2} + \cdots + qe_{q-1}x.
	\end{equation}
	One can see that $F(x)  = F_0(x) + mp^dq$ is $q$-Eisenstein for any prime $p \neq q$, because $m \not\equiv 0\pmod q$. Hence, $F(x)$ is an irreducible polynomial.
	\par 
	Next, we calculate the discriminant $\triangle(F)$, of the polynomial $F(x)$. We have:
	\begin{align}
		| \triangle(F)| & = |R(F,F')| = |R(F',F)| = |R(G,F)| \notag \\ 
		& = |q^q \prod_{i = 1}^{q_0 - 1} P(-i q_2 m \cdot (q-1)!) \cdot  \prod_{j = 1}^{q_1 -1 } F(-j m \cdot (q-1)!) |  \notag   \\
		& = |q^q \prod_{i = 1}^{q_0 - 1} \left( F_0 (-i q_2 m \cdot (q-1)!) + q m p^d \right) \cdot  \prod_{j = 1}^{q_1 -1 } \left( F_0 (-j m \cdot (q-1)!) + q m p^d \right)|   \notag  \\
		& = |q^q {m}^{q_0 - 1} \prod_{i = 1}^{q_0 -1} (q p^d + C_i) \cdot  {m}^{q_1 - 1} \prod_{j=1}^{q_1 -1} (q p^d + D_j)|, \label{eq-3.2}
	\end{align}
	where, $$C_i = \frac{F_0 (-i q_2 m \cdot (q-1)!)}{m} \quad \text{and } D_j = \frac{F_0 (-j  m \cdot (q-1)!)}{m}, $$ with $1 \leq i \leq q_0 -1 $ and $1 \leq j \leq q_1 -1$. We note that $C_i$'s and $D_j$'s are all integers. By \eqref{eq-3.2}, we see that $|\triangle(F)|$ is not squarefree, for any values of $d$, $m$, $p$ and $q$.
	\par 
	For each $i$, define $$h_i(x):= qx^d +C_i \in \mathbb{Z}[x] \quad \text{and} \quad k_j(x):= qx^d +D_j \in \mathbb{Z}[x],$$ and let $$f(x) = \prod_{i = 1}^{q_0 - 1} h_i(x) \cdot \prod_{j = 1}^{q_1 - 1} k_j(x).$$ 
	From Lemma \ref{lemma-1.2}, we have $C_i \neq D_j$, and consequently, $h_i(x) \neq k_j(x)$. Consider the polynomials $\widetilde{h}_i(x) = C_ix^d + q$ and $\widetilde{k}_j(x) = D_ix^d + q$. Both are $q$-Eisenstein polynomials. Therefore, by Proposition \ref{prop-2.7}, $h_i(x)$ and $h_j(x)$ are irreducible. Hence, $f(x)$ is a product of distinct irreducible factors. 
	\par 
	Now, we prove that there exist infinitely many primes $p$ such that $f(p)$ is squarefree. To establish this, we make use of Corollary \ref{coro-2.2}. We first claim that for each prime $r$, there exist $z \in (\mathbb{Z}{/} r^2 \mathbb{Z})^*$ such that $f(z) \not \equiv 0 \pmod{r^2}$. To prove this claim, it is sufficient to find $z$ with $\gcd(z,r) =1$ such that $f(z) \not \equiv 0 \pmod r$. So, we make four cases:
	\begin{description}
		\item[Case 1] $r<q$. In  this case, because $C_i \equiv 0 \pmod r$ and $D_j \equiv 0 \pmod r$ therefore, we get  $$f(1) \equiv \prod_{i = 1}^{q_0 -1} q \cdot \prod_{j=1}^{q_1 - 1} q  \pmod r \equiv q ^{q_0 + q_1 -2} \pmod r.$$
		Since, $r<q$ we obtain $f(1) \not \equiv  0 \pmod r$. Thus, $z=1$ yields $f(1) \not \equiv 0 \pmod {r^2}$.
		
		\item[Case 2]  $r = q.$ In this case, by \eqref{eq-3.6} and \eqref{eq-3.4} one can observe that $C_i \not \equiv 0 \pmod q$ and $D_j \not \equiv 0 \pmod q$. Hence, $f(x) \equiv \prod_{i =1}^{q_0 -1 } C_i  \cdot \prod_{i =1}^{q_1 -1 } D_j \not \equiv 0 \pmod q$ so, in this case any $z \in \mathbb{Z}$ such that $\gcd(z,q) = 1$ will satisfy $f(z) \not \equiv 0 \pmod q$.
		
		\item[Case 3] $q < r \leq d(q-1) + 1.$ In this case note that $C_i \equiv 0 \pmod m$ and $D_j \equiv 0 \pmod m$ also recall, $m \equiv 0 \pmod r$. Therefore, we get $$f(1) \equiv \prod_{i = 1}^{q_0 -1} q \cdot \prod_{j=1}^{q_1 - 1} q  \pmod r \equiv q ^{q_0 + q_1 -2} \pmod r.$$ 
		Now, following similarly as shown in Case 1, we obtain $f(1) \not \equiv 0 \pmod {r^2}$.
		
		\item[Case 4] $r \geq d(q-1) + 2$. In this case note that, $\deg(f) = d(q_0 + q_1 -2) = d(q-1)$ and $\mathbb{Z}/r\mathbb{Z}$ is a field, so $f(x) $ can have at most $d(q-1)$ roots in $\mathbb{Z}/r\mathbb{Z}$. Therefore, there exist at least two elements $\mathbb{Z}/r\mathbb{Z}$ such that $f(z) \not \equiv 0 \pmod r$ followed by $f(z) \not \equiv 0 \pmod {r^2}.$ This completes the proof of the case.
	\end{description}

	Therefore, by Corollary \ref{coro-2.2}, there exist (unconditionally if $d \leq 3$ and conditional on the $abc$-conjecture of number fields if $d\geq 4$) infinitely many primes $p$ such that $f(p)$ is squarefree.
	
	Among the infinitely many primes, select those greater than $m$, ensuring that $m \not\equiv 0 \pmod{p}$. Next, we prove that~$\gcd(q m, k_j(p))=1$ and $\gcd(q m, h_i(p)) = 1$. Before proceeding to the next step of the proof, recall the following facts: $m \not\equiv 0 \pmod{q}$, as established in \eqref{eq-3.6} and \eqref{eq-3.4}; $C_i \not\equiv 0 \pmod{q}$; and $D_j \not\equiv 0 \pmod{q}$. Furthermore, observe that $C_i \equiv 0 \pmod{m}$ and $D_j \equiv 0 \pmod{m}$. We now proceed to prove that $\gcd(qm, k_j(p)) = 1$.
	Let $t$ be a prime dividing $\gcd(qm, k_j(p))$. Then $t \mid qm$ and $t \mid k_j(p)$. Suppose $t = q$ then $q | D_j$, a contradiction. If $t \neq q$ then we have $k_j(p ) \equiv 0 \pmod t$ i.e., $p^dq + D_j \equiv 0 \pmod t$. Since $D_j \equiv 0 \pmod m$ we get $ p^dq \equiv 0 \pmod t$. This implies $t \mid p$. Also, $m \not  \equiv 0 \pmod p$, and hence $t=1$. Thus, $\gcd( mq, k_j(p)) = 1$. Similarly, one can prove that $\gcd( mq, h_i(p)) = 1$.
	\par
	Now, recall \eqref{eq-2.1} 
	$$ \triangle(F) = \triangle(K) \cdot [\mathbb{Z}_K : \mathbb{Z}[\theta]]^2.$$ 
	Then, \eqref{eq-3.2} yields 
	$$| \triangle(F) | = |q^q m^{q-1} f(p)| =  |\triangle(K)| \cdot [\mathbb{Z}_K : \mathbb{Z}[\theta]]^2.$$ Therefore, we conclude that  $\triangle(K) \equiv0 \pmod{f(p)}$ for any such value of $p$, because $f(p)$ is squarefree, where $K = \mathbb{Q(\theta)}$ with $ F(\theta) = 0$. Finally, since $F(x)$ is an Eisenstein polynomial with respect to $q$ and every prime divisor of $m$. Therefore, by Theorem \ref{thm-1.4}, we get $\triangle(K) \equiv 0 \pmod{q^q {\omega}^{q-1} }$. Hence, $\triangle(F) = \triangle(K)$ which implies $[\mathbb{Z}_K : \mathbb{Z}[\theta]] = 1$. This proves that $F(x)$ is monogenic.
\end{proof}
We now present an example of a monogenic polynomial arising from Theorem \ref{theorem_1}, which, to the best of our knowledge, is not documented in the existing literature.
\begin{example}
	Consider $q_0 = 3$, $q_1 = 5$, $q = 7$, $d = 2$, $m=11$, and $q_2 = 13$. We obtain
	$$
	a(x) = (x + 11 \cdot 13 \cdot 6!)(x + 2 \cdot 11 \cdot 13 \cdot 6!)
	$$ and
	$$b(x) = \prod_{j=1}^{4} (x + j \cdot 11 \cdot 6!).$$  
	The polynomial $F(x)$ is given by 
	\begin{align*}
		P(x) &=  x^7 + 2^3 \cdot 3 \cdot 5 \cdot 7^3 \cdot 11 x^6 + 2^8 \cdot 3^4 \cdot 5 \cdot 7^2 \cdot 11^2 \cdot 109x^5 + 2^{10} \cdot 3^6 \cdot 5^4 \cdot 7^2 \cdot 11^3 \cdot 137x^4 \\ 
		& + 2^{18} \cdot 3^7 \cdot 5^4 \cdot 7^2 \cdot 11^4 \cdot 17 \cdot 29 x^3 + 2^{21} \cdot 3^{10} \cdot 5^5 \cdot 7^4 \cdot 11^5 \cdot 13 x^2 \\ 
		& +  2^{28} \cdot 3^{13} \cdot 5^6 \cdot 7 \cdot 11^6 \cdot 13^2 x + 7 \cdot 11 \cdot 17.
	\end{align*}
	The discriminant of $F(x)$ satisfies
	$$
	| \triangle (F(x)) | = 7^7 \cdot 11^6 \cdot (\text{squarefree part}).
	$$
	Since $F(x)$ is an Eisenstein polynomial with respect to the primes $7$ and $11$, it follows that 
	$$
	7^7 \mid \triangle (K) \quad \text{and} \quad 11^6 \mid \triangle (K).
	$$ 
	Using the relation  
	$$
	\triangle (F(x)) = [\mathbb{Z}_K : \mathbb{Z}_K[\theta]]^2 \triangle(K),
	$$
	we conclude that
	$$
	[\mathbb{Z}_K : \mathbb{Z}_K[\theta]] = 1.
	$$
	This implies that $P(x)$ is monogenic.
\end{example}
\section{Non-Monogenic Polynomials from Stirling Numbers}
In this section we present a class of non-monogenic polynomials whose coefficients are Sterling numbers of first kind. Let $n$ and $k$ be positive integers such that $k \leq n$. The Stirling number
of the first kind, denoted by $s(n,k)$, counts the number of permutations of $n$ elements with $k $ disjoint cycles. One can also characterize $s(n,k)$ by, $$ (x)_n = x(x+1) \cdots (x+n-1) = \sum_{k=0}^{n} s(n,k) x^k.$$
\par 
From \cite[Sections 5.6 and 5.8]{comlet}, we recall the following identities on Stirling numbers of the first kind.  Let $n$ and $k$ be positive integers. Then we have:
\begin{align}
\label{eqn-new-1}s(n, n) &= 1, \\
\label{eqn-new-2}\quad s(n,1) &= (n-1)!,\\
\label{eqn-new-3} \quad s(n,n-1) &= {n \choose 2},\\
\label{eqn-new-4} \quad s(n,0) &= 0,\\
 \label{eqn-new-5} \quad s(n,k) &= 0 \  \text{if} \ k \geq n+1.
\end{align}
Next, we recall a theorem of Hong and Qiu \cite{Hong} on $p$-adic valuations of Stirling numbers of the first kind. We first recall some terminologies. For a prime number $p$, every nonzero rational number $r$ has a unique representation of the form $r= \pm p^k a/b$, where $k\in \mathbb{Z}, a, b \in \mathbb{N}$ and $\gcd(a,p)= \gcd(p,b)=\gcd(a,b)=1$. 
The $p$-adic valuation of such an $r$ is defined as $\nu_p(r)=k$. Let $\mathbb{Z}_p$ and $\mathbb{Q}_p$ denote the ring of $p$-adic integers and the field of $p$-adic numbers, respectively. Recall that, the $n$-th  Bernoulli number $B_n$ is defined by the Maclaurin series as
 $$\frac{x}{e^x -1 } = \sum_{n=0}^{\infty} B_n \frac{x^n}{n!}.$$ 
An odd prime $p$ is said to be a regular prime if $p$ does not divide the numerator of the Bernoulli numbers $B_0, \ldots, B_{p-3}$.
The following theorem of Hong and Qiu gives certain identities involving $p$-adic valuations of Stirling numbers of the first kind. 
\begin{theorem}\emph{\cite[Theorem 1.1]{Hong}} \label{thm-4.1}
	Let $p\geq 5$ be a prime and let $a$ and $k$ be integers such that $1 \leq a \leq p-1$ and $ 2 \leq k \leq ap - 2$. Then each of the following is true:
	\begin{enumerate}[label=(\roman*)]
		\item If $k \equiv \epsilon_k \pmod{p-1}$ then,  $$ \nu_p \left(s(ap,ap-k) \right) = \left(\nu_p(k) + 1 \right) \epsilon_k.$$
		\item If $ 2 \leq k \leq q(p-1) -1 $ and $k \not \equiv \epsilon_k \pmod{p-1}$ then,  $$ \nu_p \left(s(ap,ap-k) \right) \geq \left(\nu_p(k) + 1 \right) \epsilon_k + 1,$$ where the equality holds if and only if $ \nu_p\left(B_{2 \left\lfloor \frac{\langle k \rangle}{2} \right\rfloor
		}\right) = 0 $. In particular, if $p$ is regular then, $ \nu_p \left(s(ap,ap-k) \right) = \left(\nu_p(k) + 1 \right) \epsilon_k + 1.$
		\item If $a \geq 4$ and $a(p-1) + 2 \leq k \leq ap -2,$ then $$ \nu_p \left(s(ap,ap-k) \right) \geq a + k ap.$$
	\end{enumerate}
\end{theorem}
Setting $p$ as a regular prime greater than or equal to $5$ and putting $a =1$ in Theorem \ref{thm-4.1}, we obtain the following corollary.
\begin{corollary} \label{coro-4.1}
	Let $p \geq 5$ be a regular prime and let $k$ be an integer such that $2 \leq k \leq p-2.$ Then $\nu_p \left(s(p,p-k) \right) = \epsilon_k + 1$, where $\epsilon_k = 1 $ if $k$ is an odd integer and $\epsilon_k = 0 $ if $k$ is an even integer.
\end{corollary}
		Let $\phi(x) \in \mathbb{Z}_p[x]$ be a monic polynomial whose reduction is irreducible in $\mathbb{F}_p[x]$. We denote by $\mathbb{F}_{\phi}$ the field $\frac{\mathbb{F}_p[x]}{(\phi)}$. For any monic polynomial $f(x)\in \mathbb{Z}_p[x]$, the expression 	$$f(x)=\sum_{i=0}^{t}a_i(x)\phi(x)^i $$ 
		is called the $\phi$-expression of $f(x)$, which can be obtained upon performing the Euclidean division of $f(x)$ by successive powers of $\phi$. Here, $\deg(a_i(x))< \deg(\phi(x))$ for every $i$. The $\phi$-Newton polygon of $f(x)$ with respect to $p$, denoted by $N_{\phi}(f)$, is defined to be the lower boundary convex envelope of the set of  points $\{(i,\nu_p(a_i(x))): a_i(x) \neq 0\}$ in the Euclidean plane. 
		The $\phi_i$-index $\operatorname{ind}_{\phi}(f)$ is $\operatorname{deg}(\phi)$ times the number of points with natural integer co-ordinates that lie below or on the polygon $N_{\phi}(f)$, strictly above the horizontal axis, and strictly beyond the vertical axis.
		The following theorem is due to Ore, stated in~\cite[pp. 325]{nart}.
	\begin{theorem}\emph{ \cite[Theorem of the Index]{nart}} \label{thm-4.4}
		Let $p$ be a prime and $f(x) \in \mathbb{Z}[x]$ be a monic irreducible polynomial. Let $\bar{f}(x) = \prod_{i=1}^{t} \bar{\phi}(x)^{l_i}$ be the factorization of $f(x)$ into product of distinct irreducible polynomials over $\mathbb{F}_p$. Let $N_{\phi_i}(f)$ be the $\phi_i$-Newton polygon of $f(x)$ with respect to $p$. Then, $$\nu_p(\operatorname{ind}(f(x))) \geq \sum_{i=1}^{t} \operatorname{ind}_{\phi_i}(f).$$
	\end{theorem}
In the following theorem, we present a class of non-monogenic polynomials whose coefficients are Sterling numbers of the first kind.
	\begin{theorem}\label{thm-4.6}
		Let $ p \geq 7 $ be a regular prime, and let $ s \geq 2 $ be an integer. Suppose that the polynomial
		$$
		f(x) = (x)_p - (p-1)!x + p^s
		$$
		is irreducible over $\mathbb{Q}$, where $ (x)_p = x(x-1)(x-2)\cdots(x - p + 1) $ denotes the falling factorial. Then the polynomial $ f(x) $ is not monogenic.
	\end{theorem}
We present two proofs of Theorem \ref{thm-4.6}. The first proof uses Newton polygon techniques.
	\begin{proof}[First Proof of Theorem \ref{thm-4.6}]
		We have $f(x) = p^2 + s(p, 2)x^2 + s(p, 3)x^3 + \cdots + s(p, p)x^p$. Using Corollary \ref{coro-4.1}, one can see that $p \mid s(p,k)$ for $2 \leq k \leq p-2$, therefore $\bar{f}(x) = x^p$. Let $\phi(x) = x$. Then, $N_{\phi(x)}(f)$ is the polygonal path formed by the lower edges along the convex hull of the following points:  $$(p,\nu_p\left(s(p,p)\right)) , (p-1, \nu_p\left(s(p,p-1)\right)), (p-2,\nu_p\left(s(p,p-2)\right)), \ldots, (2,\nu_p \left(s(p,2)\right)), (0, \nu_p(p^s)).$$ In view of Corollary \ref{coro-4.1}, for $2 \le k \le p - 2$, we have $\nu_p \left(s(p,p-k) \right) = \epsilon_k + 1$, where $\epsilon_k = 1 $, if $k$ is an odd integer and $\epsilon_k = 0$, if $k$ is an even integer.
		Since $s(p, p) = 1$ and $s(p, p-1) = {p \choose 2}$, we have $\nu_p(s(p, p)) = 0$ and $\nu_p(s(p, p-1)) = 1$, respectively. If $2 \leq s \leq 4$ then $N_{\phi}(f)$ consists of two sides obtained by joining the points $(0,s), (3,1)$ and $(p,0)$ (see Figure \ref{fig_1}). Similarly, if $s \geq 5$ then  $N_{\phi}(f)$ consists of three sides $S_1$, $S_2$  and $S_3$ obtained by joining the points $(0,s), (2,2), (3,1)$ and $(p,0)$ (see Figure \ref{fig:left}). In both the cases the point $(1,1)$ lies below $N_{\phi}(f)$. Therefore, $\operatorname{ind}_{\phi}(f) \geq 1$. Thus by Theorem \ref{thm-4.4}, $\operatorname{ind}(f) \geq 1$. This implies that $f(x)$ is not monogenic.
		\begin{figure}[ht]
			\centering
			\begin{minipage}[b]{0.48\textwidth}
				\centering
				\begin{tikzpicture}[scale=0.75]
					\draw [thick, ->] (-0.5,0) -- (9.5,0);
					\draw [thick, ->] (1, -1) -- (1,5);
					\vertex (3) at (1,4) [fill=black,label=left:${\scriptstyle{s}}$] {};
					\vertex (1) at (1,1.2) [fill=black,label=left:${\scriptstyle{1}}$] {};
					
					\vertex (3) at (3,0) [fill=black,label=below:${\scriptstyle{3}}$] {};
					\vertex (3) at (3,1.2) [fill=black,label=below:${\scriptstyle{(3,1)}}$] {};
					\vertex (4) at (8,0) [fill=black,label=below:${\scriptstyle{p}}$] {};
					\node (7) at (1.8,2.8) [label=right:${\scriptstyle{S_1}}$] {};
					\node (8) at (5.5,0.5) [label=above:${\scriptstyle{S_2}}$] {};
					\node (9) at (0.4,0.1) [label=below:${\scriptstyle{(0, 0)}}$] {};
					\draw (1,4) -- (3,1.2);
					\draw (3,1.2) -- (8,0);
				\end{tikzpicture}
				
				\vspace{0.5em}
				\caption{$N_{\phi}(f)$ for $2 \leq s \leq 4$}
				\label{fig_1}
			\end{minipage}	
			\hfill	
			\begin{minipage}[b]{0.48\textwidth}
				\centering
				\begin{tikzpicture}[scale=0.75]
					\draw [thick, ->] (-0.5,0) -- (9.5,0);
					\draw [thick, ->] (1, -1) -- (1,5);
					
					\vertex (1) at (1,0.8) [fill=black,label=left:${\scriptstyle{1}}$] {};
					\vertex (2) at (1,1.8) [fill=black,label=left:${\scriptstyle{2}}$] {};
					\vertex (3) at (1,4) [fill=black,label=left:${\scriptstyle{s}}$] {};
					\vertex (4) at (2.5,0) [fill=black,label=below:${\scriptstyle{2}}$] {};
					\vertex (5) at (4,0) [fill=black,label=below:${\scriptstyle{3}}$] {};
					\vertex (4) at (2.5,1.8) [fill=black,label=below:${\scriptstyle{(2,2)}}$] {};
					\vertex (4) at (4,0.8) [fill=black,label=below:${\scriptstyle{(3,1)}}$] {};
					\draw (1, 4) -- (2.5,1.8);
					\draw (2.5,1.8) -- (4,0.8);
					\draw (4,0.8) -- (7.8,0);
					\vertex (6) at (7.8,0) [fill=black,label=below:${\scriptstyle{p}}$] {};
					\node (9) at (0.4,0.1) [label=below:${\scriptstyle{(0, 0)}}$] {};
					\node (9) at (1.5,3.3) [label=right:${\scriptstyle{S_1}}$] {};
					\node (9) at (3,1.7) [label=right:${\scriptstyle{S_2}}$] {};
					\node (9) at (5,0.85) [label=right:${\scriptstyle{S_3}}$] {};
				\end{tikzpicture}
				
				\vspace{0.5em}
				\caption{$N_{\phi}(f)$ for $s \geq 5$}
				\label{fig:left}
			\end{minipage}%
		\end{figure}
	\end{proof}
	An alternative proof of Theorem \ref{thm-4.6} can be obtained using the following theorem of Jakhar and Khanduja.
	\begin{theorem} \emph{\cite[Thoerem 1.1]{sudesh}} \label{thm-4.5}
		Let $K = \mathbb{Q}(\theta)$ be an algebraic number field with $\theta$ in the ring $\mathbb{Z}_K$ of algebraic integers of $K$ having minimal polynomial $f(x)$ over $\mathbb{Q}$. Let $p$ be a prime and $i_p(f)$ denotes the highest power of $p$ dividing the index $[\mathbb{Z}_K : \mathbb{Z}[\theta]]$. Suppose that $\bar{f}(x) = \bar{\phi}_1(x)^{e_1} \cdots \bar{\phi}_r(x)^{e_r} $ is the factorization of $\bar{f}(x)$ by replacing each coefficient of $f(x)$ modulo $p$ into a product of powers of distinct irreducible polynomials
		over $\mathbb{Z}/ p \mathbb{Z}$ with $\phi_i(x) \in \mathbb{Z}[x]$ monic. Let $t_i \geq 0$ denote the highest power $\bar{\phi_i}(x)$ dividing the polynomial $\bar{N}(x)$, where $N(x)$ belongs to $\mathbb{Z}[x]$ is defined by $f(x) =  \prod_{i=1}^{r} \phi_i(x) + p^l N(x)$, $l\geq 1$, $\bar{N}(x) \neq \bar{0}$. Let $u_i $ stands for the non negative integer given by$$
		u_i = \left\{
		\begin{array}{ll}
			\frac{(e_i -1)l + \operatorname{gcd}(e_i, l+1)-1}{2} & \text{if } t_i > \frac{e_i}{l+1},  \\
			\operatorname{max}\{lt_i,\frac{(e_i -1)(l-1) + \operatorname{gcd}(e_i, l)-1}{2}  \}, & \text{otherwise.}
		\end{array}
		\right.
		$$
		Then the following hold:
		\begin{enumerate}
			\item $i_p(f) \geq \sum_{i=1}^{r} u_i \operatorname{deg}(\phi_i(x))$.
			\item If $\gcd (e_j, l) =1 $ and $t_j = 0$ whenever $e_j > 1$, then equality holds in $(1)$, i.e., $i_p(f) = \sum_{i=1}^{r} \frac{(e_i - 1)(l-1)}{2} \deg (\phi_i(x)) $.
		\end{enumerate}
	\end{theorem}
	We now present an alternative proof of Theorem \ref{thm-4.6}.
	\begin{proof}[Alternate proof of Theorem \ref{thm-4.6}]
		We have, $$f(x) = x^p + s(p,p-1)x^{p-1} + s(p,p-2) x^{p-2} + \cdots + s(p,2)x^2 + p^s.$$ In view of Corollary \ref{coro-4.1} and \eqref{eqn-new-1}-\eqref{eqn-new-5}, we have $\nu_p \left(s(p,p-1) \right) = 1$  and $\nu_p \left(s(p,p-k) \right) = \epsilon_k + 1$, where $\epsilon_k = 1 $, if $k$ is an odd integer and $\epsilon_k = 0$, if $k$ is an even integer. Therefore, we can write,
		\begin{align}
			f(x) & = x^p + p N(x),
		\end{align}
		where $ N(x)  = \frac{s(p,p-1) }{p}x^{p-1} + \frac{s(p,p-2)}{p} x^{p-2} + \cdots+ \frac{s(p,2)}{p} x^2 + p^{s-1}.$ 
		Since, we have $p^2 \mid s(p,2), p^2 \nmid s(p,3), p^2 \nmid s(p,p-1)$ and $s \geq 2$ we get 
		\begin{align}
			\bar{f}(x) & = x^p, \\
			\bar{N}(x) & =\left(\frac{s(p,p-1)}{p} \pmod p\right)x^{p-1}+ \cdots + \left(\frac{s(p,3)}{p} \pmod p\right)x^{3} \\
						& = x^3 \left(\left(\frac{s(p,p-1)}{p} \pmod p\right)x^{p-4}+ \cdots + \left(\frac{s(p,3)}{p} \pmod p\right)\right). \nonumber
		\end{align}
		According to the notation in Theorem \ref{thm-4.5}, we get $e_1=p, r = 1, l=1,  t_1 = 3.$ Since, $t_1 = 3 \geq e_1/ (l+1)
		 = p/2$, we have $u_1 = 3.$ Thus Theorem \ref{thm-4.5} says that $i_p(f) \geq u_1 \deg(\phi_1(x)) = 3$ for $s \geq 2$. Hence, $f(x)$ is not monogenic.
	\end{proof}
	\begin{remark}
		In Theorem \ref{thm-4.6}, we need the irreducibility of $f(x) = (x)_p - (p-1)!x + p^s$ over $\mathbb{Q}$. Though we couldn't prove the irreducibility, computational experiments for primes up to 100 using \texttt{SageMath} suggest that this polynomial is irreducible over $\mathbb{Q}$. 
	\end{remark}


\end{document}